\documentclass[12pt,a4paper]{amsart}
\usepackage{amsmath,amssymb,amsbsy,amsfonts,amsthm,latexsym}
\newtheorem{lem}{Lemma}
\newtheorem{lemma}[lem]{Lemma}

\newtheorem{thm}{Theorem}
\newtheorem{theorem}[thm]{Theorem}

\newtheorem{cor}{Corollary}
\newtheorem{corollary}[cor]{Corollary}

\def\\{\cr}
\def\({\left(}
\def\){\right)}
\def\[{\left[}
\def\]{\right]}
\def\<{\langle}
\def\>{\rangle}

\def\Z{{\mathbb Z}}

\def\Q{{\mathbb Q}}

\def\C{{\mathbb C}}

\begin{document}

\title{Carmichael numbers in the sequence $\{2^n k+1\}_{n\ge 1}$}
\author{Javier Cilleruelo}
\address{Javier Cilleruelo, Instituto de Ciencias Matem\'{a}ticas (CSIC-UAM-UC3M-UCM) and
Departamento de Matem\'{a}ticas,
Universidad Aut\'onoma de Madrid,
28049, Madrid, Espa\~na
}\email{franciscojavier.cilleruelo@uam.es}
\author{ Florian Luca}
\address{Florian Luca, Fundaci\'on Marcos Moshinsky, UNAM,
Circuito Exterior, C.U., Apdo. Postal 70-543,
Mexico D.F. 04510, Mexico}\email{fluca@matmor.unam.mx}
\author{Amalia Pizarro}
\address{Amalia Pizarro, Departamento de Matem\'aticas,
Universidad de Valparaiso, Chile}\email{amalia.pizarro@uv.cl}

\begin{abstract}
We prove that for each odd number $k$, the sequence $\{k2^n+1\}_{n\ge 1}$ contains  only a finite number of  Carmichael numbers. We also prove that $k=27$ is the smallest value for which such a sequence contains some Carmichael number.
\end{abstract}
\date{\today}

\pagenumbering{arabic}

\maketitle

\section{Introduction}  The study of the presence of the prime numbers in the sequences of the form \linebreak $\{2^nk+1\}_{n\ge 1}$ is an old and difficult problem. While it is known that there exists infinitely many values of $k$, called Sierpinski numbers, for which the sequence does not contain any prime number, it is believed that for other values of $k$ the sequence $\{2^nk+1\}_{n\ge 1}$ contains, indeed, infinitely many of them.

Being unable to make any progress  on this problem, we have been successful to prove that each sequence of the above form contains only a finite number of  Carmichael numbers, which are, in a certain sense, the composite numbers most similar to the prime numbers.

A Carmichael number is a positive integer $N$ which is composite and
the congruence $a^N\equiv a\pmod N$ holds for all integers $a$, as happens when $N$ is a prime number. The
smallest Carmichael number is $N=561$ and was found by Carmichael in
1910 in \cite{Car}. It is well--known that there are infinitely many
Carmichael numbers (see \cite{AGP}), and it is believed that they are quite dense, namely that there are more than $x^{1-\epsilon}$ of them less than $x$ for every fixed $\epsilon>0$ once $x$ is large enough. Here, we let $k$ be any odd
positive integer and study the presence of Carmichael numbers in the
sequence of general term $2^n k+1$. It is known \cite{Wr}, that
the sequence $2^n+1$ does not contain Carmichael numbers, so we will
assume that $k\ge 3$ through the paper.
 We have the following result.

 For a positive
integer $m$ let $\tau(m)$ be the number of positive divisors of $m$.
We also write $\omega(m)$ for the number of distinct prime factors
of $m$. For a positive real number $x$ we write $\log x$ for its
natural logarithm.
\begin{theorem}
\label{thm:1} Let $k\ge 3$ be an odd integer. If $N=2^n k+1$ is
Carmichael,  then
\begin{equation}
\label{eq:1} n<2^{2\times 10^7 \tau(k)^2 (\log k)^2 \omega(k)}.
\end{equation}
\end{theorem}

The proof of Theorem \ref{thm:1}, which is our main result, combines combinatorial arguments with two deep tools: a quantitative version of the
Subspace Theorem as well as lower bounds for linear forms in
logarithms of algebraic numbers.

Besides $k=1$ there are other values of $k$ for which the
sequence $2^nk+1$ does not contain any Carmichael numbers. Indeed, in
\cite{BFLS}, it has been shown, among other things, that if we put
$$
{\mathcal K}=\{k:\ (2^nk+1)_{n\ge 0} \text{ contains some Carmichael
number}\},
$$
then ${\mathcal K}$ is of asymptotic density zero. This contrasts with the known fact, proved by Erd\H os and Odlyzko \cite{EO}, that the set
$$
\{k:\  (2^nk+1)_{n\ge 0} \text{ contains some prime number}\}
$$
is of positive lower density. Since $1729=2^6\times
27+1$ is a Carmichael number, we have that $27\in {\mathcal K}$.
While Theorem \ref{thm:1} gives us an upper bound on the largest
possible $n$ such that $2^n k+1$ is Carmichael, it is not useful in
practice to check if a given $k$ belongs to $\mathcal K$. For the sake of the completeness, we
prove by elementary means the following result.

\begin{theorem}
\label{thm:2}
The smallest element of ${\mathcal K}$ is $27$.
\end{theorem}

For the proofs of Theorems \ref{thm:1} and \ref{thm:2}, we start with
some elementary preliminary considerations concerning prime factors
of Carmichael numbers of the form $2^n k+1$, namely Lemmas
\ref{lem:1}, \ref{lem:2}, \ref{lem:3} and \ref{lem:4}. Then we move
on to the proofs of Theorem \ref{thm:1} and \ref{thm:2}.

\section{Preliminary considerations}

Here we collect some results about prime factors of Carmichael
numbers of the form $2^n k+1$.  There is no lack of generality in
assuming that $k$ is odd. We start by recalling Korselt's criterion.

\begin{lemma}
\label{lem:1}
$N$ is Carmichael if and only if $N$ is composite, squarefree and $p-1\mid N-1$ for all prime factors $p$ of $N$.
\end{lemma}

Assume now that  $k$ is fixed and $N=2^n k+1$ is a Carmichael number for some $n$. By Lemma \ref{lem:1}, it follows that
\begin{equation}
\label{eq:main}
2^n k+1=\prod_{i=1}^s (2^{m_i} d_i+1),
\end{equation}
where $s\ge 2$, $1\le m_i\le n$  and $d_i$ are divisors of $k$ such that $p_i=2^{m_i} d_i+1$ is prime for $i=1,\ldots,s$.
The prime factors $p=2^m d+1$ of $N$ for which $d=1$ are called
Fermat primes. For them, we must have $m=2^{\alpha}$ for some
integer $\alpha\ge 0$. The next result shows that one can bound the
Fermat prime factors of $2^n k+1$ in terms of $k$.

\begin{lemma}
\label{lem:2}
If $k\ge 3$ is odd and $p=2^{2^{\alpha}}+1$ is a prime factor of the positive integer $N=2^n k+1$, then $p<k^2$.
\end{lemma}

\begin{proof}
If $\alpha=0$, then $p=3<k^2$ because $k\ge 3$. So, we assume that $\alpha\ge 1$. We write $n=2^{\alpha} q+r$, where $|r|\le 2^{\alpha-1}$.
Then
$$
N=2^{n} k+1=2^{2^{\alpha} q+r} k+1\equiv (-1)^q 2^{r} k+1\pmod p.
$$
It then follows easily that $p$ divides one of $2^{|r|} k\pm 1$ or $k\pm 2^{|r|}$ according to the parity of $q$ and the sign of $r$. None of the above expressions is zero and the maximum such expression is $2^{|r|} k+1$. Hence, $p\le 2^{|r|}k+1\le 2^{2^{\alpha-1}} k+1$, which implies
$2^{2^{\alpha-1}}\le k$, so $2^{2^{\alpha}}\le k^2$. Clearly, the inequality is in fact strict since the left--hand side is even and the right--hand side is odd, so $p=2^{2^{\alpha}}+1\le k^2$, and the inequality is again strict since $p$ is prime and $k^2$ isn't, which completes the proof of the lemma.
\end{proof}

Primes factors $p=2^m d+1$ of $N$ for which $2^n k$ and $2^m d$ are
multiplicatively dependent play a peculiar role in the subsequent
argument. In what follows, we prove that there can be at most one
such prime factor.

\begin{lemma}
\label{lem:3} Assume that $p=2^m d+1$ is a proper prime divisor of
the integer $N=2^n k+1$, such that  $d\mid k$ and $2^m d$ and $2^n
k$ are multiplicatively dependent. Then $p\le 2^{n/3} k^{1/3}+1$.
Furthermore $N$ has at most a prime factor $p$ such that $p-1$ and
$N-1$ are multiplicatively dependent.
\end{lemma}

\begin{proof}
Let $\rho$ be the minimal positive integer such that $2^n k=\rho^u$
for some positive integer $u$. Since $2^m d$ and $2^n k$ are
multiplicatively dependent, it follows that $2^m d=\rho^v$ for some
positive integer $v$. Since $2^m d<2^n k$, it follows that $v<u$.
Furthermore, $\rho^v\equiv -1\pmod p$ and also $\rho^u\equiv -1\pmod
p$.  This implies easily that $\nu_2(u)=\nu_2(v)$, where $\nu_p(m)$
denotes the exponent of the prime $p$ in the factorization of $m$.
To see this, write $u=2^{\alpha_u}u_1,\ v=2^{\alpha_v}v_1$ with
$u_1,v_1$ odd integers and assume, for example, that
$\alpha_u<\alpha_v$. We get a contradiction observing that
$$-1\equiv  \rho^{vu_1}\equiv
(\rho^{2^{\alpha_u}u_1v_1})^{2^{\alpha_v-\alpha_u}}\equiv \left
(\rho^{uv_1}\right )^{2^{\alpha_v-\alpha_u}}\equiv 1\pmod p.$$

Writing $\alpha=\nu_2(u)=\nu_2(v)$, we get that $u=2^{\alpha}
u_1,~v=2^{\alpha} v_1$ for some odd integers $u_1$ and $v_1$.
Furthermore, since $p=(\rho^{2^{\alpha}})^{v_1}+1$ is prime, it
follows that $v_1=1$, otherwise $p$ would have $\rho^{2^{\alpha}}+1$
as a proper factor. This shows that $p$ is uniquely determined in
terms of $2^n k$. Furthermore, since $u_1\ge 3$, we get that
$\rho^{2^{\alpha}}\le (2^n k)^{1/3}$, so $p\le 2^{n/3} k^{1/3}+1$.
\end{proof}

The next lemma shows that each of the prime factors  $p=2^m d+1$ of the Carmichael number $N=2^n k+1$ for which
$2^m d $ and $2^n k $ are multiplicatively independent is small.

\begin{lemma}
\label{lem:4} Assume that $p=2^m d+1$ is a prime divisor of the
Carmichael number $N=2^n k+1$ such that $d>1$ and $2^n k$ and $2^m
d$ are multiplicatively independent. Then
$$
m<7 {\sqrt{n\log k}}\qquad  {\text{whenever}}\qquad n>3\log k.
$$
\end{lemma}

\begin{proof}
Let $p=d2^m+1$ be the prime factor of $k2^n+1$. Put $X=n/\log k$. Consider the congruences
\begin{equation}
\label{eq:congs}
d2^m\equiv -1\pmod p\qquad {\text{\rm and}}\qquad k2^n\equiv -1\pmod p.
\end{equation}
Look at the set of numbers
$$
\{mu+nv:(u,v)\in \{0,1,\ldots,\lfloor X^{1/2}\rfloor\} \}.
$$
All the  numbers in the above set are in the interval $[0,2nX^{1/2}]$ and there are $(\lfloor X^{1/2}\rfloor+1)^2>X$ of them. Thus, there exist $(u_1,v_1)\ne (u_2,v_2)$ such that
$$
|(mu_1+nv_1)-(mu_2+nv_2)|\le \frac{2nX^{1/2}}{X-1}<\frac{3n}{X^{1/2}}=3{\sqrt{n\log k}}
$$
provided that $X>3$, which is equivalent to $n>3\log k$.
We put $u=u_1-u_2$ and $v=v_1-v_2$. Then
\begin{equation}
\label{eq:impo}
(u,v)\ne (0,0),\quad  \max\{|u|,|v|\}\le X^{1/2}\quad {\text{\rm and}}\quad |um+vn|\le 3{\sqrt{n\log k}}.
\end{equation}
We may also assume that $\gcd(u,v)=1$, otherwise we may replace
the pair $(u,v)$ by the pair $(u/\gcd(u,v)),v/\gcd(u,v))$ and then all inequalities \eqref{eq:impo} are still satisfied.
In the system of congruences \eqref{eq:congs}, we exponentiate the first one to $u$ and the second one to $v$ and multiply the resulting congruences getting
$$
2^{um+vn}d^uk^v\equiv (-1)^{u+v}\pmod p.
$$
Thus, $p$ divides the numerator of the rational number
\begin{equation}
\label{eq:signs}
2^{um+vn} d^{u} k^{v}-(-1)^{u+v}.
\end{equation}
Let us see that the expression appearing at \eqref{eq:signs} above
is not zero. Assume that it is. Then, since $k$ and $d$ are odd, we
get that $um+vn=0$, $d^uk^v=1$ and $u+v$ is even. In particular,
$(2^m d)^u (2^n k)^v=1$, which is false because $(u,v)\ne (0,0)$ and
$2^n k$ and $2^m d$ are multiplicatively independent. Thus, the
expression \eqref{eq:signs} is nonzero. Since $p$ is a divisor of
the numerator of the nonzero rational number shown at
\eqref{eq:signs}, we get, by using also \eqref{eq:impo}, that
\begin{eqnarray}
\label{eq:boundforp}
p& \le & 2^{|um+vn|}d^{|u|}k^{|v|}+1  \le  2^{1+3{\sqrt{n\log k}}} k^{2X^{1/2}}\nonumber\\
& = & 2^{1+\left(3+2/\log 2\right){\sqrt{n\log k}}}<2^{7{\sqrt{n\log k}}},
\end{eqnarray}
because $2/\log 2<3$, which implies the desired conclusion.
\end{proof}

\section{The Quantitative Subspace Theorem}

We need a quantitative version of the Subspace Theorem due to
Evertse \cite{Ev}. Let us recall it. Let $M_{\Q}$ be all the places of $\Q$;
i.e. the ordinary absolute value and the p-adic absolute value. For
$y\in \Q$ and $w\in M_{\Q}$ we put $|y|_w=|y|$ if $w=\infty$ and
$|y|_w= p^{-\nu_p(y)}$ if $w$ corresponds to the prime number $p$.
When $y=0$, we set $\nu_p(y)=\infty$ and $|y|_w=0$. Then
$$
\prod_{y\in M_{\Q}} |y|_w=1\qquad {\text{\rm holds~for~all}}\qquad y\in \Q^*.
$$
Let $M\ge 2$ be a positive integer and define the {\it height} of the rational vector ${\bf y}=(y_1,\ldots,y_M)\in \Q^M$ as follows. For $w\in M_{\Q}$ write
$$
|{\bf y}|_w=\left\{ \begin{matrix} \left(\sum_{i=1}^M y_i^2\right)^{1/2} & {\text{\rm if}} & w=\infty;\\
\max\{|y_1|_w,\ldots,|y_M|_w\} & {\text{\rm if}} & w<\infty.\end{matrix}\right.
$$
Set
$$
{\mathcal H}({\bf y})=\prod_{w\in M_{\Q}} |{\bf y}|_w.
$$
For a linear form $L({\bf y})=\sum_{i=1}^M a_iy_i$ with ${\bf
a}=(a_1,\ldots,a_M)\in \Q^M$, we write ${\mathcal H}(L)={\mathcal
H}({\bf a})$.

\begin{theorem}[Evertse, \cite{Ev}]
\label{thm:EV}Let ${\mathcal S}$ be a finite subset of $M_{\Q}$
of cardinality $s$ containing the infinite place and for every $w\in
{\mathcal S}$ we let $L_{1,w},\ldots,L_{M,w}$ be $M$ linearly
independent linear forms in $M$ indeterminates whose coefficients in
$\Q$ satisfy
\begin{equation}
\label{eq:HHH} {\mathcal H}(L_{i,w})\le H\qquad {\text{\rm
for}}\quad i=1,\ldots,M\quad {\text{\rm and}}\quad w\in {\mathcal
S}.
\end{equation}
Let $0<\delta<1$ and consider the inequality
\begin{equation}
\label{eq:DP}
\prod_{w\in {\mathcal S}}\prod_{i=1}^M \frac{|L_{i,w}({\bf y})|_w}{|{\bf y}|_w}<\left(\prod_{w\in {\mathcal S}} |{\text{\rm det}}(L_{1,w},\ldots,L_{M,w})|_w\right) {\mathcal H}({\bf y})^{-M-\delta}.
\end{equation}
There exist linear subspaces
$T_1,\ldots,T_{t_1}$ of $\Q^M$ with
\begin{equation}
\label{eq:t1}
t_1\le \left(2^{60M^2} \delta^{-7M}\right)^s,
\end{equation}
such that every solution ${\bf y}\in \Q^N\backslash \{0\}$ of \eqref{eq:DP} satisfying ${\mathcal H}({\bf y})\ge H$ belongs to $T_1\bigcup \cdots \bigcup T_{t_1}$.
\end{theorem}

We shall apply Theorem \ref{thm:EV} to a certain finite subset of ${\mathcal S}$ of $M_{\Q}$ and certain systems of linear forms $L_{i,w}$ with $i=1,\ldots,M$
and $w\in {\mathcal S}$. Moreover, in our case the points ${\bf y}$ for which \eqref{eq:DP} holds are in $(\Z^*)^M$.
In particular $|{\bf y}|_w\le 1$ will hold for
all finite $w\in M_{\Q}$, as well as the inequalities
$$
1\le {\mathcal H}({\bf y})\le \prod_{w\in {\mathcal S}} |{\bf
y}|_w\le M\max\{|y_i|:i=1,\ldots,M\}.
$$
Finally, our linear forms will have integer coefficients and will in fact satisfy
\begin{equation}
\label{eq:det}
{\text{\rm det}}(L_{1,w},\ldots,L_{M,w})=\pm 1\qquad {\text{\rm for~all}}\qquad w\in {\mathcal S}.
\end{equation}
With these conditions, the following is a straightforward consequence of Theorem \ref{thm:EV} above.

\begin{corollary}
\label{cor}
Assume that \eqref{eq:det} is satisfied, that $0<\delta<1$, and consider the inequality
\begin{equation}
\label{eq:DP1}
\prod_{w\in {\mathcal S}} \prod_{i=1}^M |L_{i,w}({\bf y})|_w<M^{-\delta} \left(\max\{|y_i|:i=1,\ldots,M\}\right)^{-\delta}
\end{equation}
for some ${\bf y}\in (\Z^*)^M$. Then the conclusion  of Theorem
\ref{thm:EV} holds.
 \end{corollary}

\section{$S$-units on curves}

 We shall also use a result concerning bounds on the number of solutions of a certain type of ${\mathcal S}$-unit equation.
 Recall that an ${\mathcal S}$-{\it unit} is a non-zero rational number $y$ such that $|y|_w=1$ for all $w\not\in {\mathcal S}$. The following result is a corollary of Theorem 1.1 in \cite{Poi}.

 \begin{theorem}[Pontreau]
 \label{thm:subspace}
 Let $f(X,Y)\in \Q[X,Y]$ be a polynomial of degree $D$ which is irreducible (over ${\mathbb C}$) and which is not a binomial (i.e., has more than two nonzero coefficients).
 Then the number of solutions  $(u,v)$ of the equation
 \begin{equation}
 \label{eq:fuv}
 f(u,v)=0\qquad with\qquad (u,v)\in {\mathcal S}^2
 \end{equation}
 is bounded above by
 \begin{equation}
 \label{eq:t3}
 t_2\le 2^{104s+51} D^{6s+3} (\log(D+2))^{10s+6}.
 \end{equation}
  \end{theorem}

 \section{Baker's linear form in logarithms}

We need the following theorem due to Matveev (see \cite{Matveev} or Theorem 9.4 in \cite{Bug}).

\begin{theorem}
\label{thm:Matveev}
Let $t\ge 2$ be an integer, $\gamma_1, \ldots, \gamma_t$ be integers larger than $1$ and $b_1,\ldots,b_t$ be integers. Put
$$
B=\max\{|b_1|,\ldots,|b_t|\},
$$
and
$$
\Lambda=\gamma_1^{b_1}\cdots\gamma_t^{b_t}-1.
$$
Then, assuming that $\Lambda\neq 0$, we have
$$
|\Lambda|>\exp\left(-1.4\times 30^{t+3}\times t^{4.5}(1+\log B)(\log\gamma_1)( \log \gamma_2)\cdots (\log\gamma_t)\right).
$$
\end{theorem}

\section{Proof of Theorem \ref{thm:1}}

Since Theorem \ref{thm:2} is in fact independent of Theorem \ref{thm:1}, we shall assume that $k\ge 27$ whenever $N=2^n k+1$ is Carmichael. In particular, $\log k>3$.

{From} now on we assume that
\begin{equation}
\label{eq:n0}
n>3\log k.
\end{equation}
In particular, Lemma \ref{lem:4} holds.

We put $\delta_0=(2{\sqrt{\tau(k)}})^{-1}$ and split the prime factors of the Carmichael number $N=2^n k+1$ into four subsets as follows:
\begin{itemize}
\item[(1)] Fermat primes;
\item[(2)] The (at most one) prime $p=2^m d+1$ such that $2^m d$ and $2^n k$ are multiplicatively dependent;
\item[(3)] The primes $p=2^m d+1$ not of type (1) or (2) above with $m<\delta_0 {\sqrt{n}}$;
\item[(4)] The remaining primes.
\end{itemize}

We write $N_i$ for the product of the primes of type $i$ above for $i=1,2,3,4$. We next find an upper bound for $N_1N_2N_3$. Clearly, writing $p=2^{2^{\alpha}}+1$ for the maximal Fermat prime factor of $N$, we have that
\begin{equation}
\label{eq:N1}
N_1\le \prod_{\beta=0}^{\alpha} (2^{2^{\beta}}+1)=2^{2^{\alpha+1}}-1=(p-1)^2-1<k^4,
\end{equation}
by Lemma \ref{lem:2}. Secondly,
\begin{equation}
\label{eq:N2}
N_2\le 2^{n/3} k^{1/3}+1<2^{n/3} k,
\end{equation}
by Lemma \ref{lem:3}. Further, putting $n_0=\delta_0 {\sqrt{n}}$, we
have
\begin{eqnarray}
\label{eq:N31}
N_3 & \le & \prod_{\substack{1\le m\le n_0\\ d\mid k}} (2^m d+1)\le \prod_{1\le m\le n_0} \prod_{d\mid k} 2^{m+1} d=\prod_{1\le m\le n_0} 2^{(m+1)\tau(k)} k^{\tau(k)/2}\nonumber\\
& \le & 2^{(n_0+1)(n_0+2)\tau(k)/2+n_0 \tau(k)\log k},
\end{eqnarray}
where we used the fact that $1/(2\log 2)<1$. Assume that the exponent of $2$ in \eqref{eq:N31} is at most $n_0^2 \tau(k)=n/4$. This happens if
$$
(n_0+1)(n_0+2)\tau(k)/2+n_0 \tau(k)\log k\le n_0^2 \tau(k),
$$
which is equivalent to
$$
2n_0 \log k<n_0^2-3n_0-2.
$$
Assuming that $n_0\ge 2$, the above inequality is implied by $n_0\ge 4+2\log k$, and since $\log k>3$, the last two inequalities are satisfied when $n_0>4\log k$. Recalling the definition of $n_0$, we deduce that if
\begin{equation}
\label{eq:n}
n>64 \tau(k) (\log k)^2,
\end{equation}
then \eqref{eq:N31} implies that
\begin{equation}
\label{eq:N3}
N_3<2^{n/4}.
\end{equation}
So, if inequality \eqref{eq:n} holds, then by estimates \eqref{eq:N1}, \eqref{eq:N2} and \eqref{eq:N3}, we get
$$
N_1N_2N_3<k^4 (2^{n/3} k) 2^{n/4}=2^{7n/12} k^5<2^{7n/12+10\log k}<2^{2n/3},
$$
where the last inequality follows because $5/\log 2<10$ and $n>120
\log k$, where the last inequality is implied by \eqref{eq:n}. Since
$N_1N_2N_3N_4=N>2^n$, we get that $N_4>2^{n/3}$. On the other hand,
by Lemma \ref{lem:4}, we have that if $p\mid N_4$, then
$$
p<2^{7{\sqrt{n\log k}}} k+1\le 2^{1+2\log k+7{\sqrt{n\log
k}}}<2^{8{\sqrt{n\log k}}},
$$
where the last inequality above is a consequence of \eqref{eq:n}. Hence,
$$
2^{n/3}<N_4<2^{8 \omega(N_4){\sqrt{n\log k}}},
$$
showing that
$$
\omega(N_4)>\frac{{\sqrt{n}}}{24{\sqrt{\log k}}}.
$$
We record what we have proved as follows.

\begin{lemma}
\label{lem:6}
Assume that
\begin{equation}
\label{eq:maxn}
n>64 \tau(k)(\log k)^2.
\end{equation}
Then there exist at least ${\sqrt{n}}/(24{\sqrt{\log k}})$ primes $p=2^m d+1$ dividing $2^n k+1$ subject to the following properties:
\begin{itemize}
\item[(1)] $d>1$ is a divisor of $k$;
\item[(2)] $\delta_0 {\sqrt{n}}<m<7{\sqrt{ n\log k}};$
\item[(3)] $2^m d$ and $2^n k$ are multiplicatively independent.
\end{itemize}
\end{lemma}
We next take a look at prime divisors $p=d2^m+1$ of $N_4$. As we have seen, they have the property that
\begin{equation}
\label{eq:large}
m>n_0=\delta_0 {\sqrt{n}}.
\end{equation}
Write
\begin{equation}
\label{eq:divwrem}
n=qm+r,\quad {\text{\rm where}}\quad 0\le r\le m-1<7{\sqrt{n\log k}}.
\end{equation}
Then
\begin{equation}
\label{eq:boundforq} q=\left\lfloor \frac{n}{m}\right\rfloor\le
\frac{n}{m}\le \delta_0^{-1} {\sqrt{n}}\le 2\sqrt{\tau(k)n}.
\end{equation}
In congruences
$$
k2^{mq+r}\equiv -1\pmod p\qquad {\text{\rm and}}\qquad d2^m\equiv -1\pmod p,
$$
raise the second one to power $q$ and divide it out of the first one to get
$$
k2^r d^{-q}\equiv (-1)^{q-1} \pmod p.
$$
Thus, $p$ divides $d^q+(-1)^q k2^r$. Let us check that this last expression is nonzero. If it were zero, we would then get that $r=0$, that $q$ is odd, and that $k=d^q$, therefore $2^n k=(2^m d)^q$, which is impossible since $2^n k$ and $2^m d$ are multiplicatively independent.
Thus, $d^q+(-1)^q k2^r\ne 0$, and
\begin{eqnarray*}
|d^q+(-1)^{q} k2^r| & \le & 2^rd^q k\le 2^rk^{q+1}= 2^{r+(q+1)(\log
k)/(\log 2)}.
\end{eqnarray*}
Using \eqref{eq:divwrem} and  \eqref{eq:boundforq} we have that
\begin{eqnarray*}r+(q+1)\frac{\log k}{\log 2}&\le &7\sqrt{n\log
k}+2(\sqrt{\tau(k)n}+1)(\log k)/(\log 2)\\
&=&\frac{2(\log k)\sqrt{\tau(k)n}}{\log 2}\left (1+\frac{7\log
2}{\sqrt{\tau(k)\log k}}+\frac 1{\sqrt{\tau(k)n}}\right )\\
&<&\frac{2(\log k)\sqrt{\tau(k)n}}{\log 2}\left (1+\frac{7\log
2}{\sqrt{\tau(k)\log k}}+\frac 1{8\tau(k)\log k}\right )\\
&<&\frac{2(\log k)\sqrt{\tau(k)n}}{\log 2}\left (1+\frac{7\log
2}{\sqrt{2\log(27)}}+\frac 1{16\log(27)}\right )\\
&<&10(\log k)\sqrt{\tau(k)n}\end{eqnarray*}

Thus, writing $\delta_1=10(\log k){\sqrt{\tau(k)}}$, $U=d2^m$ and
$V=d^q+(-1)^q k2^r$, we have
$$
2^{\delta_0 {\sqrt{n}}}<U\qquad {\text{\rm and}}\qquad |V|<2^{\delta_1 {\sqrt{n}}},
$$
therefore
\begin{equation}
\label{eq:important} U>|V|^{\delta_2},\qquad {\text{\rm
where}}\qquad \delta_2={\delta_0}{\delta_1}^{-1}=(20 \tau(k)\log k
)^{-1}.
\end{equation}
We record the following conclusion.

\begin{lemma}
\label{lem:7}
Assume that inequality \eqref{eq:maxn} is satisfied. Then the number of triples of integers $(U,V_1,V_2)$ with the following properties:
\begin{itemize}
\item[(1)] $U=d2^m,\ V_1=d^q,\ V_2=(-1)^q k 2^r$;
\item[(2)] $d>1$ is a divisor of $k$ and $q$ and $r$ are nonnegative integers;
\item[(3)] $2^m d$ and $2^{mq+r} k$ are multiplicatively independent;
\item[(4)] $U+1\mid V_1+V_2$;
\item[(5)] $U>|V_1+V_2|^{\delta_2}$;
\end{itemize}
exceeds
$$
\frac{\sqrt{n}}{24{\sqrt{\log k}}}.
$$
\end{lemma}

 We next find an upper bound for the number of triples $(U,V_1,V_2)$ with the conditions (1)--(5) of Lemma \ref{lem:7} above
 in terms of $k$ alone.

 \begin{lemma}
 \label{lem:8}
 Assume that
 \begin{equation}
 \label{eq:maxn22}
 n>10^{28} (\log k)^6 \tau(k).
\end{equation}
 Then the number of triples $(U,V_1,V_2)$ with the conditions (1)--(5) of Lemma \ref{lem:7} is at most
 $$
 2^{3\times 121^{3} \tau(k)^2 (\log k)^2 \omega(k)}.
 $$
 \end{lemma}

\begin{proof}
We apply Corollary \ref{cor}. We fix the numbers $k$ and $n$. The finite set of valuations is
$$
{\mathcal S}=\{p\mid 2 k\}\cup \{\infty\},
$$
so $s=\omega(k)+2$, where we recall that $\omega(m)$ is the number of distinct prime factors of the positive integer $m$. The following argument based on the Subspace Theorem is not new. It has appeared before in \cite{BCZ}, \cite{BL}, \cite{CZ1}, \cite{CZ2}, \cite{HL},
and perhaps elsewhere. Recall that
$$
U=d2^m, \qquad V_1=d^q\qquad V_2=(-1)^q k2^r.
$$
Start with
$$
\frac{1}{U+1}=\frac{1}{U(1+1/U)}=\frac{1}{U}\left(1-\frac{1}{U}+\cdots+\frac{(-1)^{M_1-1}}{U^{M_1-1}}+\frac{\zeta_U}{U^{M_1}}\right),
$$
where $M_1$ is a sufficiently large positive integer to be determined later and $|\zeta_U|\le 2$. Thus, we get
$$
\left|\frac{1}{1+U}-\frac{1}{U}+\cdots+\frac{(-1)^{M_1}}{U^{M_1}}\right|<\frac{2}{U^{M_1+1}}.
$$
Multiply the above inequality by $V=V_1+V_2$, to get
$$
\left|\frac{V}{1+U}-\frac{V_1+V_2}{U}+\cdots+\frac{(-1)^{M_1}(V_1+V_2)}{U^{M_1}}\right|\le \frac{2|V|}{U^{M_1+1}}.
$$
Multiply both sides above by $U^{M_1}$ to get
\begin{equation}
\label{eq:subspace}
\left|\frac{VU^{M_1}}{1+U}-V_1U^{M_1-1}-V_2U^{M_1-1}+\cdots+(-1)^{M_1}V_1+(-1)^{M_1}V_2\right|\le \frac{2|V|}{U}.
\end{equation}
We take $M=2M_1+1$ and label the $M$ variables as
$$
{\bf y}=(y_1,\ldots,y_{2M_1+1})=(z,y_{1,M_1-1},y_{2,M_1-1},\ldots,y_{1,0},y_{2,0}).
$$
We take the linear forms to be
$$
L_{1,\infty}({\bf y})=z-y_{1,M_1-1}-y_{2,M_1-1}+\cdots+(-1)^{M_1} y_{1,0}+(-1)^{M_1}y_{2,0}
$$
and $L_{i,w}({\bf y})=y_i$ for $(i,w)\ne (1,\infty)$. It is clear that these forms are linearly independent for every fixed $w\in {\mathcal S}$, and condition \eqref{eq:det} is satisfied for  them. We evaluate the double product
\begin{equation}
\label{eq:DP2}
\prod_{w\in {\mathcal S}} \prod_{i=1}^M |L_{i,w}({\bf y})|_w,
\end{equation}
when $(U,V_1,V_2)$ are as in Lemma \ref{lem:7},
$$
z=\frac{(V_1+V_2)U^{M_1}}{1+U}\quad {\text{\rm and}}\quad y_{i,j}=V_iU^{j}\quad (i=1,2,~j=0,\ldots,M_1-1).
$$
For $i\ge 2$, $y_i$ is an ${\mathcal S}$-unit and $L_{i,w}({\bf y})=y_i$ for all $w\in {\mathcal S}$, so that
\begin{equation}
\label{eq:S1}
\prod_{w\in {\mathcal S}} \prod_{i=2}^M |L_{i,w}({\bf y})|_w=1.
\end{equation}
For $i=1$, since $V/(1+U)\in \Z$, it follows that $z$ is an integer multiple of $U^{M_1}$. Hence,
\begin{equation}
\label{eq:S2}
\prod_{w\in {\mathcal S}\backslash \{\infty\}} |L_{1,w}({\bf y})|_w\le U^{-M_1}.
\end{equation}
Finally, we have
\begin{equation}
\label{eq:S3}
|L_{1,\infty}({\bf y})|_{\infty}\le \frac{2|V|}{U},
\end{equation}
by \eqref{eq:subspace}. Multiplying \eqref{eq:S1}, \eqref{eq:S2} and \eqref{eq:S3}, we get that
\begin{equation}
\label{eq:DP3}
\prod_{w\in {\mathcal S}} \prod_{i=1}^M  |L_{i,w}({\bf y})|_w\le \frac{2|V|}{U^{M_1+1}}.
\end{equation}
Choose $M_1=\lfloor 3/\delta_2\rfloor$. Then we have that $M_1>{2}/{\delta_2}$, therefore
$$
U^{M_1}>U^{2/\delta_2}>|V|^2,
$$
by \eqref{eq:important}. Thus,
\begin{equation}
\label{eq:Vandmax}
\frac{2|V|}{U^{M_1+1}}<\frac{|V|}{U^{M_1}}\le \frac{1}{|V|}.
\end{equation}
We now compare $|V|$ and $|V_i|$ for $i=1,2$. If $q$ is even, then $V=|V_1|+|V_2|$. Assume now that $q$ is odd. Then
\begin{equation}
\label{eq:111}
|V|=|V_1||k2^rd^{-q}-1|.
\end{equation}
By using the inequality of Theorem \ref{thm:Matveev} with $t=3$,
$\gamma_1=k,~\gamma_2=2,~\gamma_3=d,~b_1=1,~b_2=r,~b_3=-q$, we have
that
\begin{equation}
\label{eq:t}
|k2^rd^{-q}-1|>\exp(-c_1 (\log k)^2\log n),
\end{equation}
where we used the fact that $\max\{d,k\}\le k$ and $\max\{r,q\}\le n$, and we can take $c_1=1.4\times 30^6\times 3^{4.5} \times
2\times \log 2$. Let us check that
\begin{equation}
\label{eq:xxx}
|k2^r d^{-q}-1|>U^{-1}.
\end{equation}
For this, since $U>2^m>2^{\delta_0 {\sqrt{n}}}$, it is enough that
$$
\delta_0 (\log 2){\sqrt{n}}>c_1 (\log k)^2 \log n,
$$
which is equivalent to
\begin{equation}
\label{eq:n2} \frac{\sqrt{n}}{\log(\sqrt n)}>c_2(\log
k)^2{\sqrt{\tau(k)}},
\end{equation}
where $c_2=11.2\times 30^6\times 3^{4.5}$. Let us spend some time
unraveling \eqref{eq:n2}. It is  easy to prove that if $A>3$ then
the inequality
$$
\frac{x}{\log(x)}>A\quad {\text{\rm is~ implied ~by}}\quad x>2A\log
A.
$$
Using this argument it follows that it suffices that
\begin{equation}
\label{eq:yyy} \sqrt n>2c_2(\log k)^2\sqrt{\tau(k)}\log\left (
c_2(\log k)^2{\sqrt{\tau(k)}} \right )
\end{equation}
Since $2\log \log k<\log k$, $\tau(k)<k$ and $\log(c_2)<28$, we get
that
$$
\log(c_2)+(\log \tau(k))/2+2\log \log k<28+1.5\log k<11\log k,
$$
where the last inequality follows because $\log k>3$. Hence, in order for \eqref{eq:yyy} to hold, it suffices that
$$
{\sqrt{n}}>22c_2 (\log k)^3 {\sqrt{\tau(k)}},
$$
which is satisfied for
\begin{equation}
\label{eq:maxn2}
n>10^{28} (\log k)^6 \tau(k),
\end{equation}
which is exactly condition \eqref{eq:maxn22}. Since condition \eqref{eq:maxn22} holds, we get that also inequality \eqref{eq:xxx} holds. With \eqref{eq:111}, we get that
$$
|V|=|V_1||k2^rd^{-q}-1|>|V_1|U^{-1},
$$
therefore $|V_1|<|V|U<|V|^2$. A similar argument shows that $|V_2|\le V^2$. Thus, we always have $\max\{|V_1|,|V_2|\}\le |V|^2$ regardless of the parity of $q$.
Hence,
\begin{eqnarray*}
|V_i|U^{M_1-1} & \le & |V|^2 U^{M_1-1}\le |V|^{M_1+1}\qquad (i=1,2);\\
\frac{|V|U^{M_1}}{1+U} & < & |V|U^{M_1-1}<|V|^{M_1+1}.
\end{eqnarray*}
This shows that for our vector ${\bf y}$ we have that
\begin{equation}
\label{eq:max}
\max\{|y_i|:i=1,\ldots,M\}<|V|^{M_1+1}.
\end{equation}
Finally, we have
$$
M=2M_1+1\le \frac{6}{\delta_2}+1<120\tau(k)\log k+1<2^{\delta_0\sqrt
n}<U<|V|.
$$
Indeed, the middle inequality  is equivalent to
$$n>\tau(k)(2\log2)^2\log^2(120\tau(k)\log k+1),$$ which is implied by
\eqref{eq:maxn2}. Thus,
$$
M\max\{|y_i|:i=1,\ldots,M\}<|V|^{M_1+2}.
$$
Comparing \eqref{eq:DP3} with \eqref{eq:Vandmax} and the last estimate above, we get
\begin{equation}
\label{eq:suitable}
\prod_{w\in {\mathcal S}} \prod_{i=1}^M |L_{i,w}({\bf y})|_w\le \frac{2|V|}{U^{M_1+1}}\le \frac{1}{|V|}\le \left(M\max\{|y_i|:i=1,\ldots, M\}\right)^{-\delta},
\end{equation}
where $\delta=1/(M_1+2)$.

We now apply Corollary \ref{cor}  with $H={\sqrt{M}}$. Note that relation
\eqref{eq:HHH} holds for our system  of forms. Let us check the condition
${\mathcal H}({\bf y})\ge 1$ for our ${\bf y}\in \Z^M$. Observe that since the last two coordinates of ${\bf y}$ are $V_1=d^q$ and $V_2=(-1)^q k 2^r$, it follows that $|{\bf y}|_2=1$ and
$$
\prod_{w\in {\mathcal S}\backslash \{2,\infty\}} |{\bf y}|_w\ge k^{-1}.
$$
Thus, taking into account just the contribution of $y_3=V_2 U^{M_1-1}$ to $|{\bf y}|_{\infty}$, we get that
\begin{eqnarray*}
{\mathcal H}({\bf y}) & \ge &  |y_3|_{\infty} \prod_{w\in {\mathcal S}\backslash \{\infty\}} |{\bf y}|_w\ge V_2 U^{M_1-1} k^{-1}\\
& \ge &  U^{M_1-1}\ge
6^{M_1-1}>{\sqrt{2M_1+1}},
\end{eqnarray*}
where the last inequality holds for all $M_1\ge 2$, which is certainly the case for us since
$M_1>2/\delta_2=40\tau(k)\log k>80$. Hence, all conditions from
Corollary \ref{cor} are satisfied. We get that all solutions ${\bf y}$ of our
problem lie in $t_1$ proper subspaces of $\Q$, where $t_1$ is
bounded as in \eqref{eq:t1}.

Let us take such a subspace. We then get an equation of the form
\begin{equation}
\label{eq:TTT}
d_0\left(\frac{V_1+V_2}{1+U}\right) U^{M_1}+\sum_{i=1}^{2}\sum_{j=0}^{M_1-1} c_{i,j} V_iU^{j}=0
\end{equation}
for some vector of coefficients
$$
(d_0,c_{i,j}:1\le i\le 2, 0\le j\le M_1-1)\in \Q^{M}
$$
not all zero. We divide across equation \eqref{eq:TTT} by $V_1U^{-M_1}$. Further, by setting $W=V_2/V_1=(-1)^qk2^r d^{-q}$, we arrive at
$$
d_0\frac{W+1}{U+1}+\sum_{i=1}^2 \sum_{j=0}^{M_1-1} c_{i,j} W^{i-1} U^{-(M_1-j)}=0.
$$
The last equation above is a rational function in the pair $(U,W)$, which is nonzero as a rational function (this has been checked in many places,
like \cite{BCZ}, or \cite{CZ2}, for example). Clearing the denominator $1+U$, we arrive at an equation of the form
\begin{equation}
\label{eq:poly}
\sum_{i=0}^1 \sum_{j=0}^{M_1} e_{i,j} W^i U^{-j}=0
\end{equation}
for some coefficients $(e_{i,j}:0\le i\le 1,0\le j\le M_1)\in \Q^{M}$, not all zero. Put $U_1=U^{-1}$. The above equation \eqref{eq:poly} is of the form
$$
WP(U_1)+Q(U_1)=0,
$$
where $P(X)$ and $Q(X)$ are in $\Q[X]$ of degrees at most $M_1$. We distinguish a few cases.

When $P(X)=0$, then $Q(X)\ne 0$. Then
$U_1$ has at most $M_1$ values, therefore $m$ is determined in at most $M_1$ ways.

A similar argument works when $Q(X)=0$.

Assume now that
none of $P(X)$ and $Q(X)$ is the constant zero polynomial. Put
$$
F(X,Y)=YP(X)+Q(X).
$$
Then any solution $(U,W)$ to equation \eqref{eq:poly} leads to a solution to the equation $F(U_1,W)=0$. Assume next that $F(X,Y)$ is a binomial polynomial.
It then follows that $P(X)=c_1X^{f_1}$ and $Q(X)=c_2X^{f_2}$ for some nonzero rational coefficients $c_1,c_2$ and some
nonnegative integer exponents $f_1,f_2$. Then since $F(U_1,W)=0$, it follows that $WU^{f_2-f_1}=-c_2/c_1$ is uniquely determined. To recover $W$ and $U$ uniquely,  we need to check that $W$ and $U$ are multiplicatively independent. If they were not, we would have integers $\lambda$ and $\mu$ not both zero such that
$$
|W|^{\lambda}=k^{\lambda}2^{r\lambda} d^{-q\lambda}=d^{\mu}2^{m\mu}=U^{\mu}.
$$
Hence, we get that $r\lambda-m \mu=0$, and that $k^{\lambda}=d^{\mu+\lambda}$. If $\lambda=0$, we then get that $d^{\mu}=1$, so $\mu=0$, therefore
$(\lambda,\mu)=0$, which is false. Thus, $\lambda\ne 0$. This leads easily to the conclusion that $2^n k$ and $2^m d$ are multiplicatively dependent (in fact, we get the relation $(2^m d)^{\mu+q\lambda}=(k2^n)^\lambda$), which is not the case.
Thus, when $F(X,Y)$ is a binomial polynomial, then there is at most one convenient solution to $F(U_1,W)=0$.

Assume now that $F(X,Y)$ has at least three nonzero coefficients. Write $P(X)=X^{f_1} P_1(X)$ and $Q(X)=X^{f_2}Q_1(X)$, where $f_1,~f_2$ are nonnegative integer exponents, and $P_1(X)$ and $Q_1(X)$ are polynomials in $\Q[X]$ with
$P_1(0)Q_1(0)\ne 0$. Replace $F(X,Y)$ by
$$
\frac{F(X,Y)}{X^{\min\{f_1,f_2\}}}=YX^{f_1-\min\{f_1,f_2\}} P_1(X)+X^{f_2-\min\{f_1,f_2\}} Q_1(X).
$$
Then any solution $(U,W)$ to equation \eqref{eq:poly} still satisfies $F(U_1,V)=0$ with this new $F(X,Y)$ (because $U_1\ne 0$). Furthermore,
$F(X,Y)$ is now irreducible over $\C[X,Y]$ because it is not divisible by neither $X$ nor $Y$ and it is linear in $Y$. Its degree $D$ satisfies
$$
D\le \max\{1+{\rm deg}(P_1(X), {\rm deg}(P_2(X)\}\le M_1+1<M.
$$
But then, by Theorem \ref{thm:subspace}, the number of solutions $(U,W)$ is at most
\begin{equation}
\label{eq:t3new} t_2\le 2^{104s+51} M^{6s+3} (\log(M+2))^{10s+6}.
\end{equation}
Recall that $s=\omega(k)+2$. Note that $U$ determines uniquely $d$ and $m$, which in turn determine also $q$ and $r$ uniquely by \eqref{eq:divwrem}.
To summarize, we get that for fixed $n$ satisfying \eqref{eq:maxn2} and odd $k\ge 3$, the number of triples $(U,V_1,V_2)$ with the conditions (1)--(5) of Lemma \ref{lem:7} is at most
$$
t_1t_2,
$$
where $t_1$ and $t_2$ are shown at \eqref{eq:t1} and
\eqref{eq:t3new}, respectively. We now bound $t_1$ and $t_2$ for our
application.

Note that since $\delta^{-1}=M_1+2,\ M=2M_1+1$  and $M_1=\lfloor
3/\delta_2\rfloor$, we get easily that
\begin{eqnarray}\delta^{-1}&=&(M+3)/2\nonumber\\
M & \le & \frac{6}{\delta_2}+1\le 121 \tau(k){\log k},\label{omega}\\
s & = & \omega(k)+2\le 3\omega(k).\label{s}
\end{eqnarray}
Therefore
\begin{eqnarray}
\label{eq:boundfort1}
t_1 & < & (2^{60M^2} \delta^{-7M})^s\nonumber\\
& < & (2^{60M^2}\left ((M+3)/2\right )^{7M})^s\nonumber;
\end{eqnarray}
and since $s\ge 3$,
\begin{eqnarray}
\label{eq:boundont3} t_2 & < & 2^{104s+51}
M^{6s+3}(\log(M+2))^{10s+6}\\&<& \left ( 2^{221}M^7
\log^7(M+2)\right )^s\nonumber.
\end{eqnarray}
Hence,
\begin{eqnarray}
t_1t_2&<&\left (2^{60M^2\left (1+\frac 1{60}\left
(\frac{7\log((M+3)/2)}{(\log 2)M} +\frac{221}{M^2}+\frac{7\log
M}{(\log 2)M^2}+\frac{7\log \log(M+2)}{(\log 2)M^2}\right )\right )}
\right )^s\nonumber \\
&<&2^{61sM^2}\label{t1t3}
\end{eqnarray}
provided the quantity
$$E(M)=\frac{7\log((M+3)/2)}{(\log 2)M}
+\frac{221}{M^2}+\frac{7\log M}{(\log 2)M^2}+\frac{7\log
\log(M+2)}{(\log 2)M^2}
$$
satisfies $E(M)<1$. We observe that
\begin{eqnarray*}
M & = & 2M_1+1=2 \lfloor 3/\delta_2\rfloor +1=2\lfloor 60\tau(k)\log
k\rfloor +1\\
& \ge & 2\lfloor 60\times 2\times\log (27)\rfloor +1 = 791
\end{eqnarray*}
and
certainly, $E(M)<1$ for $M\ge 791$.

Finally, putting \eqref{omega} and \eqref{s} in \eqref{t1t3} we get
$$t_1t_2<2^{3\times 121^3\omega(k)\tau^2(k)\log^2k}.$$
\end{proof}

Theorem \ref{thm:1} follows now from Lemmas \ref{lem:7} and \ref{lem:8}. Indeed, observe first that inequality \eqref{eq:maxn22} implies inequality \eqref{eq:maxn}. Next, assuming that inequality \eqref{eq:maxn22}, the conclusion of Lemmas \ref{lem:7} and \ref{lem:8} is that
\begin{eqnarray*}
n & < & 24^2 (\log k) 2^{6\times  121^{3} \tau(k)^2 (\log k)^2 \omega(k)}\\
& < & 2^{2\times  10^{7} \tau(k)^2 (\log k)^2 \omega(k)},
\end{eqnarray*}
where we have used that $24^2 (\log k)<2^{\tau(k)^2 (\log k)^2
\omega(k)}$ for $k\ge 27$.

 So, to finish, it suffices to prove that
$$
2^{2 \times 10^7 \tau(k)^2 (\log k)^2 \omega(k)}>10^{28} (\log k)^6
\tau(k),
$$
which follows since $2^x>x^4$ for $x>100$ with
$$
x=2\times 10^7
\tau(k)^2 (\log k)^2 \omega(k).
$$

\section{The proof of Theorem \ref{thm:2}}

We have to show that if $k\le 25$ is odd, then there is no Carmichael number of the form $2^n k+1$. We distinguish five cases, according to whether $k$ is prime, or $k\in \{9,15,21,25\}$.

\subsection{$k\le 23$ is prime}

By Lemma \ref{lem:1}, we have that if $p$ is a Fermat prime factor of $N=2^n k+1$, then $p< k^2\le 23^2$, therefore $p\in \{3,5,17,257\}$.
By the Main Theorem 2 in \cite{Wr}, we get that there are only seven possibilities for $N$, namely
\begin{eqnarray}
\label{eq:list}
N & \in & \left\{5\times 13\times 17, 5\times 13\times 193\times 257, 5\times 13\times 193\times 257\times 769,\right. \\
&& \left. 3\times 11\times 17, 5\times 17\times 29, 5\times 17\times 29\times 113, 5\times 17\times 257\times 509\right\}.\nonumber
\end{eqnarray}
There is another possibility listed in \cite{Wr}, namely
$$
N=5\times 29\times 113\times 65537\times 114689,
$$
which is not convenient for us since $65537$ is a Fermat number exceeding $23^2$. However, no number from list \eqref{eq:list} is of the form $2^n k+1$ for some odd prime $k\le 23$.

\subsection{Preliminary remarks about the cases $k\in \{9,15,21,25\}$}

We first run a search showing that there is no Carmichael number of the form $2^n k+1$ for all $n\in \{1,\ldots,256\}$. Suppose now that $n>256$. Write
$$
2^n k+1=\prod_{i=1}^s (2^{m_i} d_i+1)
$$
where $1\le m_i\le n$, $d_i\mid k$ for $i=1,\ldots,s$ and $p_i=2^{m_i} d_i+1$ is prime for all $i=1,\dots,s$. We assume that the primes are listed in such a way that
$$
a=m_1\le m_2\le \cdots.
$$
We first show that $n>a+20$. Indeed, assume that this is not so. If $p_1$ is a Fermat prime, then, by Lemma \ref{lem:1}, we have $a\le (\log k)/\log 2<5$, so $n\le a+20\le 25$, which is false. If $2^n k$ and $2^{m_1} d_1$ are multiplicatively dependent, then Lemma \ref{lem:2} shows that $a\le n/3$. Thus, $n\le a+20\le n/3+20$, therefore $n\le 30$, which is again false. Finally, assume that $d_1>1$ and $2^{m_1} d_1$ and $2^n k $ are multiplicatively dependent. Then Lemma \ref{lem:3} shows that $a=m_1<7{\sqrt{n\log k}}<14{\sqrt{n}}$ because $3\log k\le 3\log 27<12<n$. Thus, $n<14{\sqrt{n}}+20$, which is impossible for $n\ge 256$. So, indeed $n>a+20$. From this, we conclude that if we put $b_i$
such that
$$
b_i=\nu_2(p_1p_2\cdots p_i-1)
$$
for $i=1,2,\ldots,s-1$ and $b_i\le a+20$, then $a_{i+1}\le b_i$. This argument will be used in what follows without further referencing.

\subsection{$k=9$}

If $p$ is a Fermat number dividing $N$, then $p\le 9^2=81$ by Lemma \ref{lem:1}, so $p\in \{3,5,17\}$. Clearly, $3\nmid 2^n\cdot 9+1$ for any $n\ge 1$, therefore $p\in \{5,17\}$. We now write
$$
2^n \cdot 9+1=\prod_{i=1}^s (2^{a_i}+1)\prod_{i=1}^t (2^{b_i}\cdot 3+1)\prod_{i=1}^{u} (2^{c_i} \cdot 9+1),
$$
where $a_1<\cdots<a_s,~b_1<\cdots<b_t,~c_1<\cdots<c_u$. It is easy to see that $a_1,b_1,c_1$ cannot be all three distinct. Let
$a=\min\{a_1,b_1,c_1\}$. We do a case by case analysis according to the number $a$.

If $a=1$, the possibilities are
that two of $3,7,19$ divide $N$. As we have seen, $3\nmid N$, so both $7$ and $19$ divide $N$. However, $7$ never divides $2^n\cdot 9+1$, which is a contradiction.

If $a=2$, then two of $5,~13,~37$ divide $N$. However, $5\mid N$ implies $n\equiv 0\pmod 4$. Similarly, $13\mid N$ implies
$n\equiv 10\pmod {12}$, while $37\mid N$ implies $n\equiv 2\pmod {36}$, and no two such congruences can simultaneously hold.

If $a=3$, then $2^3\cdot 3+1=25$ is not prime, and we get a contradiction.

If $a=4$, then neither one of $2^4\cdot 3+1=49=7^2$ or $2^4\cdot 9+1=145=5\times 29$ is prime, again a contradiction.

Thus, $a\ge 5$. In particular, $s=0$, and $b_1=c_1$. Put $p_1=2^a \cdot 3+1$ and $p_2=2^a \cdot 9+1$.
For an odd prime $p$ we put ${\text{\rm ord}}_p(2)$ for the multiplicative order of $2$ modulo $p$.
Then ${\text{\rm ord}}_2(p_i)=2^{\alpha_i}\cdot \delta_i$, where $\alpha_i\le a$ and $\delta_i\in \{1,3,9\}$ for $i=1,2$. The congruences
$$
2^n\cdot 9\equiv -1\pmod {p_1}\qquad {\text{\rm and}}\qquad 2^{2a} \cdot 9\equiv 1\pmod {p_1}
$$
imply $2^{n-2a}\equiv -1\pmod {p_1}$, which implies that ${\text{\rm ord}}_{p_1}(2)\mid 2n-4a$, therefore
$2n\equiv 4a \pmod {2^{\alpha_1}}$. Similarly, from the congruences
$$
2^n \cdot 9\equiv -1\pmod {p_2}\quad {\text{\rm and}}\quad 2^{a}\cdot 9\equiv -1\pmod {p_2},
$$
we get $2^{n-a}\equiv 1\pmod {p_2}$, so $n\equiv a\pmod {2^{\alpha_2}}$, or $4n\equiv 4a\pmod {2^{\alpha_2}}$. Thus, putting
$\alpha=\min\{\alpha_1,\alpha_2\}$, we get that $2n\equiv 4a\pmod {2^{\alpha}}$ and also $4n\equiv 4a\pmod {2^{\alpha}}$, therefore
$2n\equiv 0\pmod {2^{\alpha}}$. In particular, $2^{\alpha}\cdot 9\mid 18n$, showing that one of the numbers  $p_1$ or $p_2$ divides
$2^{18n}-1$. Since
$$
p_i\mid 2^n\cdot 9+1\mid 2^{18n} \cdot 9^{18}-1=(2^{18n}-1)9^{18}+(9^{18}-1)
$$
for both $i=1,2$, we get that one of $p_1$ or $p_2$ divides
$$
9^{18}-1=2^4\cdot 5\cdot 7\cdot 13\cdot 19\cdot 37\cdot 73\cdot 757\cdot 530713.
$$
However, none of the primes appearing in the right hand side above is of the form $2^a \cdot 3+1$ for some $a\ge 5$, which completes the argument in this case.

\subsection{$k=15$}

If $p$ is a Fermat number dividing $N$, then $p<15^2$, therefore $p\in \{3,5,17\}$. Clearly, it is not possible that $3\mid 2^n\cdot 15+1$
or $5\mid 2^n\cdot 15+1$ for any $n\ge 1$, so only $p=17$ is possible. We write
$$
2^n\cdot 15+1=\prod_{i=1}^s (2^{m_i} d_i+1),
$$
where $s\ge 2$, $d_i\mid 15$ for $i=1,\ldots,s$ and $p_i=2^{m_i} d_i+1$ is prime for all $i=1,\ldots,s$. We put again
$a=\min\{m_i:i=1,\ldots,s\}$.  Then $p_1=2^a d_1+1$ and $p_2=2^a d_2+1$ are both prime divisors of
$N$ for two different divisors $d_1$ and $d_2$ of $15$.   We again do a case by case analysis according to the values of $a$.

If $a=1$, then $p_1,~p_2\in \{7,11,31\}$. However, $7\nmid 2^n\cdot 15+1$ for any $n\ge 1$, therefore both $11$ and $31$ divide
$N$. However, $11\mid N$ implies that $n\equiv 3\pmod {10}$, while $31\mid N$ implies that $n\equiv 1\pmod 5$, and these two congruences
are contradictory.

Assume next that $a=2$. Since $2^2\cdot 5+1=21=3\times 7$ is not prime, it follows that the only possibility is that both
$13$ and $61$ divide $N$. However, the condition $13\mid N$ implies that $n\equiv 5\pmod {12}$, whereas $61\mid N$ implies that
$n\equiv 2\pmod {60}$, and again the last two congruences for $n$ are contradictory.

The case $a=3$ is not possible since neither $2^3\cdot 3+1=25=5^2$ nor $2^3\cdot 15+1=121=11^2$ is prime.

Assume now that $a=4$. Since $2^4\cdot 3+1=49=7^2$ and $2^4\cdot 5+1=81=3^4$, it follows that the only possibility is that both
$17$ and $241$ divide $N$.  However, the condition $17\mid N$ implies that $n\equiv 7\pmod 8$, whereas $241\mid N$ implies that
$n\equiv 4\pmod {24}$, and these last congruences are again contradictory.

The case $a=5$ is also impossible since none of $2^5\cdot 5+1=161=7\times 23$ and $2^5\cdot 15+1=13\times 37$ is prime.

So, from now on $a_i\ge 6$ for all $i=1,\ldots,s$. Let $p=2^b d+1$ for some $b\ge 6$. Assume that $d=5$. Since $p\equiv 1\pmod 8$, it follows that $(-1/p)=(2/p)=1$, where the above notation is the Legendre symbol. Since $5\equiv -2^{-b}\pmod p$, it follows that $(5/p)=1$. Since
$3\equiv -2^{-n} \times 5^{-1} \pmod p$, it follows that $(3/p)=1$, therefore, by quadratic reciprocity, $(p/3)=1$, therefore $p\equiv 1\pmod 3$. However, $2^b\cdot 5+1$ is never $1 \pmod 3$ for any positive integer $b$. This shows that $d\ne 5$. In particular, $d\in \{3,15\}$
for all prime factors $p$ of $N$. Assume next that $d=3$. By a similar argument, we have $(-1/p)=(2/p)=(3/p)=1$ and now the condition
$5\equiv -2^{-n} \times 3^{-1}\pmod p$ implies that $(5/p)=1$, which, via quadratic reciprocity, implies that $p\equiv 1,4\pmod  5$. Since also $p=2^b \cdot 3+1$, it follows easily that $b\equiv 0\pmod 4$ (for the values of $b$ congruent to $1,2,3$ modulo $4$ we get that
$2^b\cdot 3+1$ is congruent to $2,3,0$ modulo $5$, respectively, none of which is convenient). Since when $b\equiv 1\pmod 3$, we have $2^b\cdot 3+1$ is a multiple of $7$, we get that $b\equiv 0,2\pmod 3$, which together with the fact that $b\equiv 0\pmod 4$, leads to $b\equiv
0,~8\pmod {12}$.

Suppose first that $a\equiv 0\pmod {12}$. It then follows that the smallest $b>a$ such that $2^b\cdot 3+1$ is a prime factor of $N$ is
$b\ge a+8$. Write $p_1=2^a\cdot 3+1$ and $p_2=2^a \cdot 15+1$. Then
$$
p_1p_2=1+2^{a+1}(9+2^{a-1} \cdot 45)
$$
is a divisor of $N$. So, $p_3=2^{a+1} \cdot 15+1$ is also a divisor of $N$. Thus,
\begin{eqnarray*}
p_1p_2p_3 & = & (1+2^{a+1}(9+2^{a-1} 45))(1+2^{a+1} \cdot 15)\\
& = & 1+2^{a+1}(24+2^{a-1}\cdot 45)+2^{2a+2} \cdot 15(9+2^{a-1}\cdot 45)\\
& = &
1+2^{a+4}(3+2^{a-4} M_1)
\end{eqnarray*}
is a divisor of $N$, where $M_1$ is some odd integer. Thus, $p_4=2^{a+4}\cdot 15+1$ is also a prime factor of $N$. We then have
\begin{eqnarray*}
p_1p_2p_3p_4 & = & (1+2^{a+4}(3+2^{a-4}\cdot M_1))(1+2^{a+4}\cdot 15)\\
& = & 1+2^{a+4}(18+2^{a-4} M_2)\\
& = & 1+2^{a+5}(9+2^{a-5} M_2),
\end{eqnarray*}
where $M_2$ is some odd integer. Thus, $p_5=2^{a+5} \cdot 15+1$ is also a prime factor of $N$. However, since
$a\equiv 0\pmod {12}$, it follows that $a+5\equiv 5\pmod {12}$, which implies that $p_5\equiv 0\pmod {13}$, a contradiction.

Assume next that $a\equiv 8\pmod {12}$. Since $2^8\cdot 15+1=3841=23\times 167$ is not prime, it follows that $a\ge 20$.
We take again $p_1=2^a\cdot 3+1$ and $p_2=2^a\cdot 15+1$. Then
$$
p_1p_2=1+2^{a+1}(9+2^{a-1}\cdot 45)
$$
is a divisor of $N$. Thus, $p_3=2^{a+1}\cdot 15+1$ is a divisor of $N$ and
\begin{eqnarray*}
p_1p_2p_3 & = & (1+2^{a+1}\cdot 15)(1+2^{a+1}(9+2^{a-1}\cdot 45))\\
& = & 1+2^{a+1}(24+2^{a-1} M_1)\\
& = & 1+2^{a+4}(3+2^{a-4}M_1)
\end{eqnarray*}
is a divisor of $N$ for some odd integer $M_1$. Since $a+4\equiv 0\pmod {12}$, it follows that either $2^{a+4}\cdot 3+1$ is a divisor of
$N$ or $2^{a+4}\cdot 15+1$ is a divisor of $N$ but not both. In the first case, $p_4=2^{a+4}\cdot 3+1$ and
$$
p_1p_2p_3p_4=(1+2^{a+4}\cdot 3)(1+2^{a+4}(3+2^{a-4} M_1))=1+2^{a+5}(3+2^{a-5} M_2)
$$
is a divisor of $N$ for some odd integer $M_2$, while in the second case we have $p_4=2^{a+4}\cdot 15+1$ and
$$
p_1p_2p_3p_4=(1+2^{a+4}\cdot 15)(1+2^{a+4}(3+2^{a-4} M_1))=1+2^{a+5}(9+2^{a-5} M_2)
$$
is a divisor of $N$ again for some odd integer $M_2$. In both cases, we conclude that $p_5=2^{a+5}\cdot 15+1$ divides $N$ and
$$
p_1p_2p_3p_4p_5=(1+2^{a+5}\cdot 15)(1+2^{a+5}(T+2^{a-5} M_2))
$$
is a divisor of $N$ for some $T\in \{3,9\}$. We thus get that
$$
p_1p_2p_3p_4p_5\quad {\text{\rm equals}}\quad 1+2^{a+6}(9+2^{a-6} M_3)\quad {\text{\rm or}}\quad 1+2^{a+8}(3+2^{a-8} M_3)
$$
according to whether $T=3$ or $T=9$, respectively. In the first case, we have that $p_6=2^{a+6}\cdot 15+1$ divides $N$, whereas in the second case
$p_6=2^{a+8}\cdot 15+1$ divides $N$. Observe that
$$
p_1\cdots p_6=(1+2^{a+6}(9+2^{a-6} M_3))(1+2^{a+6} \cdot 15)=1+2^{a+9} (3+2^{a-9} M_4)
$$
for some odd integer $M_4$ in the first case, whereas
$$
p_1\cdots p_6=(1+2^{a+8}(3+2^{a-8} M_3))(1+2^{a+8}\cdot 15)=1+2^{a+9} (9+2^{a-9} M_4)
$$
in the second case. In either case, $p_7=2^{a+9}\cdot 15+1$ is a divisor of $N$. However, since $a\equiv 8\pmod {12}$, it follows that
$a+9\equiv 5\pmod {12}$, so $p_7$ is a multiple of $13$, which is a contradiction.

\subsection{$k=21$}

If $p$ is a Fermat factor of $N$, then $p<21^2$, therefore $p\in \{3,5,17,257\}$. Clearly, it is not possible that $3\mid 2^n\cdot 21+1$.
One also checks that $257\nmid 2^n\cdot 21+1$ for any $n\ge 1$, so only $p=5,17$ are possible. We write
$$
2^n\cdot 21+1=\prod_{i=1}^s (2^{m_i} d_i+1),
$$
where $s\ge 2$, $d_i\mid 21$ for $i=1,\ldots,s$ and $p_i=2^{m_i} d_i+1$ is prime for all $i=1,\ldots,s$. We put again
$a=\min\{m_i:i=1,\ldots,s\}$.  Then $p_1=2^a d_1+1$ and $p_2=2^a d_2+1$ are both prime divisors of
$N$ for two different divisors $d_1$ and $d_2$ of $21$.   We again do a case by case analysis according to the values of $a$.

When $a=1$, we get that two of $2+1,~2\cdot 3+1,~2\cdot 7+1,~2\cdot 21+1$ are prime factors of $N$, which is impossible because $2+1=3$ and $2\cdot 3+1=7$ cannot divide $N$ while $2\cdot 7+1=15=3\times 5$ is not prime.

When $a=2$, we get that two of $2^2+1,~2^2\cdot 3+1,~2^2\cdot 7+1,~2^2\cdot 21+1$. Since $85=5\times 17$ is not prime, it follows that
$N$ is divisible by two of $\{5,13,29\}$. If $5\mid N$, then $n\equiv 2\pmod 4$. If $13\mid N$, then $n\equiv 3\pmod {12}$, whereas if
$29\pmod N$, then $n\equiv 25\pmod {28}$, and no two of the above congruences are simultaneously possible (the last two imply that $n\equiv 3\pmod 4$  and $n\equiv 1\pmod 4$, respectively).

The case $a=3$ is not possible since neither $2^3\cdot 3+1=25=5^2$ nor $2^3\cdot 7+1=57=3\times 19$ is prime.

{From} now on, $a\ge 4$. Let $p=2^b d+1$ be a prime factor of $N$. Let us show that $d$ cannot be $7$. Assume that it is.
Since $b\ge 4$, it follows that $(-1/p)=(2/p)=1$, and since $7\equiv -2^{-b}\pmod p$, it follows that $(7/p)=1$. Since also
$3\equiv -2^{-n}\times 7^{-1}\pmod p$, it follows that $(3/p)=1$, so, by quadratic reciprocity, $p\equiv 1\pmod 3$. However, $2^b\cdot 7+1$ is never congruent to $1$ modulo $3$, which is a contradiction. Hence, $d\in \{1,3,21\}$. Further, suppose that $d=3$. Then,
by the same argument, $(-1/p)=(2/p)=1$ and so $3\equiv -2^{-b}\pmod p$, therefore $(3/p)=1$. Since also
$7\equiv -2^{-n}\times 3^{-1}\pmod p$, we get that $(7/p)=1$, which, by quadratic reciprocity, implies that $(p/7)=1$.
Since $p=2^b\cdot 3+1$, it follows that $b\equiv 0\pmod 3$ (for $b$ congruent to $1,~2$ modulo $3$ we get that $p$ is congruent to
$0,~6$ modulo $7$, and none of these possibilities is convenient). Further, in this same instance, it is clear that we cannot have $b\equiv 3\pmod 4$, since it would lead to $p=2^b\cdot 3+1$ being a multiple of $5$. Hence, $b\equiv 0,1,2\pmod 4$, which together with
$b\equiv 0\pmod 3$, leads to $b\equiv 0,6,9\pmod {12}$.

Assume now that $a=4$. Since $2^4\cdot 3+1=49=7^2$, it follows that the only possibility is that both $17$ and $337$ divide
$N$. The condition $17\mid N$ implies that $n\equiv 2\pmod 8$ while the condition that $337\mid N$ implies that $n\equiv 4\pmod {21}$.
The above conditions imply that $n\equiv 130\pmod {168}$. Further
$$
17\times 337=5729=1+2^5\times 179
$$
is a divisor of $N$. It follows that $N$ is divisible by one of $1+2^5\cdot 3=97$ or $1+2^5\cdot 21=673$. However, there is no $n\ge 0$
such that $97\mid 2^n\cdot 21+1$. Further, $673\mid N$ implies that $n\equiv 5\pmod {48}$, which is incompatible with
$n\equiv 130\pmod {168}$ since the first one means that $n\equiv 2\pmod 3$, whereas the second one means that $n\equiv 1\pmod 3$.

So, from now on we have that $a\ge 5$. Thus, $p_1=2^a\cdot 3+1$ and $p_2=2^a\cdot 21+1$. As we have seen, $a\equiv 0\pmod 3$.
It is also easy to see that $a\equiv 0,1\pmod 4$, otherwise one of $2^a \cdot 3+1$ or $2^a\cdot 21+1$ is a multiple of $5$. Thus,
$a\equiv 0,9\pmod {12}$.

Now
$$
p_1p_2=(1+2^a\cdot 3)(1+2^a\cdot 21)=1+2^a(3+21)+2^{2a} \cdot 63=1+2^{a+3}(3+2^{a-3}\cdot 63).
$$
Assume first that $a\equiv 0\pmod {12}$. Then the next prime factor of $N$ of the form $p=2^b \cdot 3+1$ must have $b\equiv 0,6,9\pmod {12}$, therefore $b\ge a+6$, so $p_3=2^{a+3}\cdot 21+1$ must divide $N$. However, since $a\equiv 0\pmod {12}$, it follows that $p_3$ is a multiple of $13$. Assume next that $a\equiv 9\pmod {12}$. In particular, $a\ge 9$. In fact, since $2^9 \cdot 3+1=29\times 53$ is not prime, it follows that $a\ge 21$.
Then none of $2^{a+1}\cdot 3+1$ and $2^{a+2}\cdot 3+1$ are prime factors of $N$ since $a+1$ and $a+2$ are not multiples of
$3$. Thus, none of $2^{a+1}\cdot 21+1$ and $2^{a+2}\cdot 21+1$ is a prime factor of $N$ either. Hence,
exactly one of $2^{a+3}\cdot 3+1$ or $2^{a+3}\cdot 21+1$ is a prime factor of $N$. Assume  that it is $p_3=2^{a+3}\cdot
21+1$. Then
$$
p_1p_2p_3=(1+2^{a+3}(3+2^{a-3}\cdot 69))(1+2^{a+3}\cdot 21)=1+2^{a+6}(3+2^{a-6} M_1)
$$
for some odd integer $M_1$. Since $a+4$ and $a+5$ are not multiples of $3$, it follows that none of $2^{a+3} \cdot 3+1$ or $2^{a+4} \cdot3+1$ are factors of $N$, therefore $2^{a+3}\cdot 21+1$ and $2^{a+4}\cdot 21+1$ are not factors of $N$ either. Hence, one of
$2^{a+6}\cdot 3+1$ or $2^{a+6}\cdot 21+1$ is a prime factor of $N$. Since $a+6\equiv 3\pmod{12}$ it follows that the first one
cannot be a prime factor of $N$, whereas the second one is a multiple of $13$ so it cannot be prime. So, assume that $p_3=2^{a+3}\cdot
3+1$. Then
$$
p_1p_2p_3=(1+2^{a+3}(3+2^{a-3}\cdot 69)(1+2^{a+3}\cdot 3)=1+2^{a+4}(3+2^{a-4} M_1)
$$
for some odd integer $M_1$. Since $a+4$ is not a multiple of $3$, it follows that $2^{a+4}\cdot 3+1$ is not a prime factor of $N$,
and so $p_4=2^{a+4} \cdot 21+1$ is a prime factor of $N$.
Observe that
$$
p_1p_2p_3p_4=(1+2^{a+4}(3+2^{a-4} M_1)(1+2^{a+4}\cdot 21)=1+2^{a+7}(3+2^{a-7} M_2)
$$
for some odd integer $M_2$. Next, $2^{a+5}\cdot 3+1$ are $2^{a+6}\cdot 3+1$ are not prime factors of $N$ because $a+5$ and $a+6$ are congruent to $2,3\pmod {12}$, so $2^{a+5}\cdot 21+1$ and $2^{a+6}\cdot 21+1$ are not prime factors of $N$ either. Thus,
one of $2^{a+7}\cdot 3+1$ and $2^{a+7}\cdot 21+1$ is a prime factor of $N$, and since $a+7$ is not a multiple of $3$, it follows that
$p_4=2^{a+7}\cdot 21+1$ is a prime factor of $N$. Now
$$
p_1p_2p_3p_4=(1+2^{a+7}(3+2^{a-7} M_2)(1+2^{a+7}\cdot 21)=1+2^{a+10}(3+2^{a-10} M_3)
$$
for some odd integer $M_3$. Since $a+8$ is not a multiple of $3$, it follows that $2^{a+8}\cdot 3+1$ does not divide $N$, therefore
$2^{a+8}\cdot 21+1$ does not divide $N$ either. If $2^{a+9}\cdot 3+1$ is a prime factor of $N$, then $2^{a+9}\cdot 21+1$ is a prime factor of $N$ also, but since $a\equiv 9\pmod {12}$, it follows that $a+9\equiv 2\pmod 4$, therefore $2^{a+9}\cdot 21+1$ is in fact a multiple of $5$. Thus,
none of $2^{a+9}\cdot 3+1$ or $2^{a+9}\cdot 21+1$ is a prime factor of $N$. Since $a+10$ is not a multiple of $3$, we get that
$2^{a+10}\cdot 3+1$ cannot be a prime factor of $N$. Thus, $p_5=2^{a+10}\cdot 21+1$ is a prime factor of $N$. Thus,
$$
p_1\cdots p_5=(1+2^{a+10}(3+2^{a-10} M_3))(1+2^{a+10} \cdot 21)=1+2^{a+13} (3+2^{a-13} M_4)
$$
is a divisor of $N$ for some odd integer $M_4$. Since $a+11$ is not a multiple of $3$, it follows that $2^{a+11}\cdot 3+1$
is not a prime factor of $N$. Therefore $2^{a+11}\cdot 21+1$ is not a prime factor  of $N$ either. As for $a+12$, it follows that either both $p_6=2^{a+12}\cdot 3+1$ and $p_7=2^{a+12}\cdot 13+1$ are prime factors of $N$, or none of them is. If both of them are, then
$$
p_6p_7=(1+2^{a+12}\cdot 3)(1+2^{a+12}\cdot 21)=1+2^{a+15}\cdot M_5
$$
for some odd integer $M_5$. So, in either case, namely when both $p_5$ and $p_6$ are prime factors of $N$, or when none of them is, we still infer that one of $2^{a+13}\cdot 3+1$ or $2^{a+13}\cdot 21+1$ is a prime factor of $N$. However, since $a\equiv 9\pmod {12}$, $a+13$ is not a multiple of $3$, so $2^{a+13}\cdot 3+1$ cannot be a prime factor of $N$, whereas since $a+13\equiv 2\pmod 4$, the number
$2^{a+13}\cdot 21+1$ is a multiple of $5$, so it cannot be a prime factor of $N$ either. This completes the analysis of the case $k=21$.

\subsection{$k=25$}

If $p$ is a Fermat number dividing $N$, then $p<25^2=625$, therefore $p\in \{3,5,17,257\}$. Clearly, $5\nmid 2^n\cdot 25+1$
for any $n\ge 0$, and one can check that $257\nmid 2^n\cdot 25+1$ for any $n\ge 0$.  Thus, $p\in \{3,17\}$.
We now write
$$
2^n \cdot 25+1=\prod_{i=1}^s (2^{a_i}+1)\prod_{i=1}^t (2^{b_i}\cdot 5+1)\prod_{i=1}^{u} (2^{c_i} \cdot 25+1),
$$
where $a_1<\cdots<a_s,~b_1<\cdots<b_t,~c_1<\cdots<c_u$. It is easy to see that $a_1,b_1,c_1$ cannot be all three distinct. Let
$a=\min\{a_1,b_1,c_1\}$. We do a case by case analysis according to the number $a$.

If $a=1$, then $2\cdot 25+1=51=3\times 17$ is not prime, so we must have that both $3$ and $11$ divide $2^n\cdot 25+1$.
If $3\mid 2^n\cdot 25+1$, then $n\equiv 1\pmod 2$, while if $11\mid 2^n\cdot 25+1$, then $n\equiv 7\pmod {10}$.
Thus, $33=2^5+1$ is a divisor of $N$. This implies that $b=\min\{a_2,b_2,c_2\}\le 5$.
Put $b=\min\{a_2,b_2,c_2\}$. Assume first that $b<5$. Then not all three $a_2,b_2,c_2$ are distinct. The case $b=2$ is not possible since $2^2+1=5$ is
not a divisor  of $N$ and  $2^2\cdot 5+1=21=3\times 7$ is not prime. The case $b=3$ is also not possible because $2^3\cdot 25+1=201=3\times 67$ is not prime.
In case $b=4$, since $2^4\cdot 5+1=81=3^4$ is not prime, the only possibility is that both $2^4+1=17$ and $2^4\cdot 25+1=401$.
However, $17\mid N$ implies that $n\equiv 1\pmod 8$, whereas $401\mid N$ implies that $n\equiv 4\pmod {200}$, and these congruences cannot be both satisfied.
Thus, $b=5$. However, this is not possible since none of $2^5\cdot 5+1=161=7\times 23$ and $2^5\cdot 25+1=801=3^2\times 89$ is prime.

Assume now that $a=2$. This is not possible because $2^2+1=5$ cannot divide $N$ and $2^2\cdot 5+1=21=3\times 7$ is not prime.

The case $a=3$ is not possible because $2^3\cdot 25+1=201=3\times 67$ is not prime.

Assume now that $a=4$. Since $2^4\cdot 5+1=81=3^4$, it follows that $N$ is divisible by both $2^4+1=17$ and $2^4\cdot 25+1=401$.
Again the condition $17\mid N$ implies that $n\equiv 1\pmod 8$, whereas $401\mid N$ implies that $n\equiv 4\pmod {200}$ and these two congruences cannot simultaneously hold.

{From} now on, $a\ge 5$, therefore both $2^a\cdot 5+1$ and $2^a\cdot 25+1$ are prime factors of $N$, which is false since one of these two numbers is always a multiple of $3$.

\medskip

{\bf Acknowledgment.} F.~L. thanks Jan-Hendrik Evertse for useful
advice and for providing some references. This work was done when
F.~L. and A.~P. visited the Mathematical Department and the ICMAT of
the UAM in Madrid, Spain, April 2012. These authors thank these
institutions for their hospitality and support. During the
preparation of this paper. F.~L. was also supported in part by
Project PAPIIT IN104512 and a Marcos Moshinsky Fellowship. J.~C. was
supported by Project MTM2011-22851 of the MICINN. A. P. was supported in part by project
Fondecyt No. 11100260.

\end{document}